\documentclass[12pt,reqno]{amsart}
\usepackage{amsmath, amsthm, amssymb}
\usepackage[T1]{fontenc}
\usepackage[OT2, T1]{fontenc}
\usepackage[russian,english]{babel}
\usepackage{geometry}
\usepackage{cancel}
\usepackage{soul}
\usepackage{comment}
\usepackage{color}
\usepackage{amscd}
\usepackage{tabularx}
\usepackage{url}
\usepackage{eurosym}
\usepackage{braket}
\usepackage{hyperref}

\usepackage{amsrefs}

\usepackage{setspace}
\usepackage{verbatim}
\usepackage{mathrsfs}
\usepackage{dsfont}
\usepackage{amsbsy}
\usepackage{amstext}
\usepackage{amsthm}
\usepackage{amssymb}
\usepackage{xcolor}
\usepackage{cancel}
\usepackage{graphicx}
\usepackage[export]{adjustbox}

\usepackage{float}
\PassOptionsToPackage{normalem}{ulem}
\usepackage[toc,page]{appendix}
\usepackage{graphicx,tikz}

\usepackage{adjustbox,graphicx,subcaption}
\makeatletter
\usepackage{eurosym}
\usepackage{mathrsfs}
\usepackage{dsfont}
\usepackage{xcolor}
\usepackage{cancel}

\newtheorem{theorem}{Theorem}[section]
\newtheorem{lemma}[theorem]{Lemma}

\theoremstyle{definition}

\theoremstyle{remark}
\newtheorem{remark}[theorem]{Remark}

\numberwithin{equation}{section}

\def \bR {\mathbb R}
\def \bbR {\mathbb R}

\def \bZ {\mathbb Z}

\def \cC {\mathcal C}

\def \cM {\mathcal M}

\def\N{{\mathbb N}}
\def\R{{\mathbb R}}
\def\Z{{\mathbb Z}}

\def \ds1 {\mathds{1}}

\def\eps{\varepsilon}

 \def\ge{\geqslant}

\newcounter{@ToDo}
\newcommand{\todo@helper}[1]{%
	({\color{blue}TODO~\arabic{@ToDo}: {#1\@addpunct{.}}})%
}
\newcommand{\todo}[1]{\stepcounter{@ToDo}%
	\relax\ifmmode\text{\todo@helper{#1}}%
	\else\todo@helper{#1}\fi%
}
\newcounter{@cdo}
\newcommand{\cdo@helper}[1]{%
	({\color{red}CITE~\arabic{@cdo}: {#1\@addpunct{.}}})%
}
\newcommand{\cdo}[1]{\stepcounter{@cdo}%
	\relax\ifmmode\text{\cdo@helper{#1}}%
	\else\cdo@helper{#1}\fi%
}

\author{Mingfeng Chen}
\author{Rajula Srivastava}
\author{Niclas Technau}

\address{Mingfeng Chen;
Institut des Hautes Études Scientifiques, 91440 Bures-sur-Yvette, France}

\email{chenmf@ihes.fr}

\address{Rajula Srivastava; 
Department of Mathematics, 
University of Wisconsin 480 Lincoln Drive, Madison, WI, 53706, USA}
\email{rsrivastava9@wisc.edu}

\address{Niclas Technau; 
Department of Mathematics, 
University of Wisconsin 480 Lincoln Drive, Madison, WI, 53706, USA}
\email{technau@wisc.edu}

\begin{document}
\title[Rational Points Near Nondegenerate Curves]
{Counterexamples for Rational Points Near Curves}
\date{\today}

\begin{abstract}
Let $\delta \in [0,1/2]$ and $Q\geq 1$.
Given a compact $C^{\infty}$-curve $\cC$ in $\mathbb{R}^n$, denote by $N_{\cC}(\delta, Q)$ 
the number of pairs $(\mathbf{a}, q)\in \mathbb{Z}^n\times [Q/2,Q]$ such that the rational point 
$\mathbf{a}/q$ is $\delta/q$-close to $\cC$.
We investigate the range of $\delta$ in terms of $Q$ which is necessary for the heuristic 
$N_{\cC}(\delta, Q)\asymp\delta^{n-1}Q^2$ to be correct for the moment curve.
For a nondegenerate curve $\cC\subset \R^n$ and for sufficiently large $Q$, Hickman and the second author previously established that $N_{\cC}(\delta, Q)\lesssim_{\nu} \delta^{n-1}Q^2$ 
for all $\delta\in[Q^{\alpha(n+1)+\nu},1/2)$ and any $\nu>0$.
Here $\alpha(n)=-4n^{-2} +O(n^{-3})$ was explicitly computed.
We show that this range is surprisingly close 
to being sharp in the asymptotic sense as $n\to \infty$,
and can at most be extended to $\delta\in[Q^{\alpha(n)},1/2)$.
In particular, the sharp exponent can only be about $O(n^{-3})$ better, 
even though it was widely believed before that an improvement of $O(n^{-1})$ should be possible. The complementary upper bound in this remaining range has been recently established by Gan--Guo--Oh. 
\end{abstract}
\maketitle
\section{Introduction}
Let $I=[-1, 1]$. 
Consider the rational points near 
the moment curve 
\begin{equation}\label{def moment curve}
\Gamma=\{(t, t^2, \ldots, t^n): t\in I\}
\end{equation}
in $\R^n$, that is
    \begin{align}\label{def counting function}
        N_{\Gamma}(\delta, Q):=\# \Big\{(a,q)\in \Z\times \N: \frac{a}{q}\in I, \ \frac{Q}{2}\leq q\leq Q, \  
        \max_{i\leq n}\Big\|q\Big(\frac{a}{q}\Big)^i\Big\|_{\bR/\bZ}\leq \delta 
        \Big\}.
    \end{align}
Heuristically, we expect $ {N}_{\Gamma}(\delta,Q)$ should be comparable to $ \delta^{n -1}Q^{2}$ whenever $\delta>Q^{-c}$ for some $c>0$. Until very recently, it was widely believed that this heuristic should be true for 
$$
\delta>Q^{-\frac{1}{n-1}}.
$$ However, in the recent work \cite{chen2026sharp}, the authors and A. Seeger showed that the conjectured upper bound  ${N}_{\Gamma}(\delta,Q)\lesssim_{\eps} \delta^{n-1} Q^{2}$  fails  if 
\begin{equation}\label{eq CSST range}
Q^{\frac{\eps-1}{n-1}}\ll \delta < Q^{-\frac{1}{2n-3}}    
\end{equation}
and $Q$ is sufficiently large.\footnote{Here and throughout the paper we use  the following notation: Given a list of objects/parameters $L$ and non-negative real numbers $A$, $B$, we write $A \lesssim_L B$ or $B \gtrsim_L A$ or $A = O_L(B)$ if $A \leq C_L B$ where $C_L > 0$ is a constant depending only on $L$ and underlying manifold $\cM$. We write $A \approx_L B$ or $A\asymp_L B$ when both $A \lesssim_L B$ and $B \lesssim_L A$ hold.}  Further, for $n=3$, this paper established the expected upper bound in the range $\delta > Q^{\varepsilon-\frac{1}{3}}$, for $\varepsilon>0$, which is optimal up to endpoints. In this note, we build on the example introduced in \cite{chen2026sharp} to sharpen the range of $\delta$ over which the heuristic upper bound can hold for $n\geq 4$. In particular, the examples detailed below yield new obstructions in dimensions $n\geq 5$ which further restrict the range of $\delta$ in which the expected upper bound may hold.

\begin{theorem}
\label{thm main}
Let $n\geq 3$. For $Q$ sufficiently large, $ N_{\Gamma}(\delta, Q)\lesssim \delta^{n-1}Q^2$
implies that $\delta \gtrsim Q^{\alpha(n)}$, where
\begin{equation}\label{eq def alpha}
 \alpha(n):=\begin{cases}
    -\frac{4}{n^2+2n-3}, &\textrm{for odd } n,\\
    -\frac{4}{n^2+2n-4}, &\textrm{for even } n.
\end{cases}
\end{equation}

\end{theorem}

Before \cite{chen2026sharp}, complete results were known in the case of  planar curves and hypersurfaces with non-vanishing curvature or under finite-type conditions (see \cite{huxley1994rational, vaughan2006diophantine, huang19, chen2025rational, huangduke, st2023}), or for certain low co-dimension surfaces  under a nonvanishing Gaussian curvature condition with respect to all normals \cite{schindler2022density, sri24, hsw25}. For lower bounds see \cite{BeresnevichDickinsonVelani07}. 

Alternatively, one can also study the more difficult problem of counting rational points with a fixed denominator $q$ near $\Gamma$, that is, 
obtaining bounds for the counting function
\begin{align}
        L_{\Gamma}(\delta, q):=\# \Big\{(a,q)\in \Z\times \N: \frac{a}{q}\in I,\  
        \max_{i\leq n}\Big\|q\Big(\frac{a}{q}\Big)^i\Big\|_{\bR/\bZ}\leq \delta 
        \Big\}.
    \end{align}
For a nondegenerate curve $\cC$ in $\bbR^3$, of which the moment curve is a prototypical example, it was shown in \cite{huang19} that
the heuristic upper bound $L_{\cC}(\delta, q)\lesssim \delta^2q$  is true when $\delta \gtrsim q^{-\frac15}(\log q)^{\frac25}$ and $q\gtrsim 1$.
In \cite{hickman2025counting}, J. Hickman and the second author generalized this result to higher dimensions. More specifically, \cite{hickman2025counting} established that if $\nu>0$ then 
\begin{equation*}
    L_{\cC}(\delta,q)\lesssim_{\nu} \,\delta^{n-1} q \textrm{ for } q\gtrsim_{\nu} 1 \textrm{ and } \delta \geq q^{\beta_{\textrm{HS}}(n)+\nu},
\end{equation*}
where
\begin{equation}
    \label{eq delta range}
    \beta_{\textrm{HS}}(n):=\alpha(n+1)=\begin{cases}
    -\frac{4}{n(n+4)}, & \textrm{ for even } n,\\
    -\frac{4}{n(n+4)-1}, & \textrm{ for odd } n.    
    \end{cases}
\end{equation}
Since $N_{\cC}(\delta, Q)=\sum_{q=\frac{Q}{2}}^{Q}L(\delta, q)$, we also obtain that
\begin{equation}
    \label{eq hs ub}
    N_{\cC}(\delta,Q)\lesssim_{\nu} \,\delta^{n-1} Q^2 \textrm{ for } Q\gg_{\nu} 1 \textrm{ and } \delta \geq Q^{\beta_{\textrm{HS}}(n)+\nu}.
\end{equation}

\begin{remark}
As a joint consequence of \eqref{eq hs ub} and Theorem \ref{thm main}, at the time when this manuscript was being written, the only range where the expected upper bound for $N_{\cC}(\delta, Q)$ still needed to be established was 
\begin{equation}
    \label{eq rem delta range}
   Q^{\alpha(n)} \leq \delta \leq Q^{\beta_{\textrm{HS}}(n)}\textrm{ for } n\geq 4.
\end{equation}
A short calculation
involving geometric series reveals that
$$\beta_{\textrm{HS}}(n)=\alpha(n+1)=-\frac{4}{n^2}+\frac{16}{n^3}+O(n^{-4}), \qquad \alpha(n)=-\frac{4}{n^2}+\frac{8}{n^3}+O(n^{-4})\,.$$
Thus, the gap between exponents in $\mathbb{R}^n$ is $O(n^{-3})$ . In particular, the argument in \cite{hickman2025counting} already captures the leading constant \(4\) in the \(n^{-2}\) scale for the exponent.

This was a pleasant surprise to us as the positive result in \cite{hickman2025counting} in all dimensions $n\geq 3$ was obtained by fairly classical and relatively simple Fourier analytic arguments. It required no arithmetic input other than a simple volume based counting in $O(1)$ neighborhoods of nested surfaces in the dual frequency space. In contrast, in the recent work \cite{chen2026sharp} for space curves in $\mathbb{R}^3$, A. Seeger and the current authors required a substantially elaborate Fourier analytic argument and the planar counting result of Vaughan--Velani from \cite{vaughan2006diophantine} to exploit additional cancellation in the denominator variable $q$ and obtain a sharp result.
In this sense, the rational point counting problem for nondegenerate curves appears to get more tractable as the dimension $n$ increases, due to additional restrictions on the range of $\delta$.
\end{remark}

\begin{remark}
While working on the manuscript, the authors were made aware of the recent work \cite{GanGooOh} of Gan--Guo--Oh, where they have established the expected upper bound on $N_{\cC}(\delta, Q)$ in the aforementioned remaining range of $\delta$, thereby extending the result from \cite{chen2026sharp} for space curves in $\mathbb{R}^3$ to nondegenerate curves in dimensions four and higher.
\end{remark}

\subsection*{\it Acknowledgment}
R.S. would like to thank Jonathan Hickman for many helpful conversations while working on \cite{hickman2025counting}, and all authors thank Andreas Seeger for his encouragement. This research was supported in part by NSF grant 2554830 (R.S.) and the University of Wisconsin–Madison, Office of 
the Vice Chancellor for Research with funding from the Wisconsin Alumni Research Foundation (R.S. and N.T.). 
Further,
N.T. was supported by a grant from the Simons Foundation [SFI-MPS-TSM-00025353].

\subsection*{\it Use of AI tools} As explained in detail below, the authors already had a heuristic and justification for the  necessary condition from Theorem \ref{thm main}. Our first strategy was to use existing arguments from \cite{Bers12, ChowSrivastavaTechnauYu2026} and build on the counter-example introduced in \cite{chen2026sharp} (see \S \ref{subsec heur}--\ref{subsec existing lit}). However, checking in detail
and recalling the set-up of \cite{ChowSrivastavaTechnauYu2026}
is more involved than 
necessary. A more elementary and self-contained argument (in particular, the recursive identity \eqref{eq master rec} adapted specifically to the moment curve $\Gamma$) was later suggested to us by Chat G.P.T. 6 Astra. This is described in \S\ref{subsec rigorous}--\ref{subsec lower}. The entire argument, including the proof of \eqref{eq master rec},  is written and verified independently by the authors, who take full responsibility for the content.

\section{A Family of Counterexamples}
\subsection{A Heuristic}
\label{subsec heur}
Recall $I=[-1, 1]$.
    Consider the rational points near 
    the moment curve $\Gamma$
    from \eqref{def moment curve}, 
    and recall the counting function $N_{\Gamma}$
    from \eqref{def counting function}.
For $1\leq m\leq n-2$ (the extreme $n-1$ will not matter as we will see later),
set 
$$
T_{m}=\left(\frac{\delta}{Q}\right)^{\frac{1}{m+1}}, 
\qquad I_m:= 
\Big[-\frac{T_m}{100}, \frac{T_m}{100} \Big]\,.
$$
For $a/q\in I_m$ and $i\geq m+1$, we have
\begin{align}
    |q(\frac{a}{q})^i|\leq Q T_m^i\leq 
    Q\frac{\delta}{Q}=\delta\,.
\end{align}
Thus,
\begin{align}
    N_{\Gamma}(\delta, Q) \ge \# \{(a,q)\in \Z\times \N: {a}/q\in I_m, \ Q/2\leq q\leq Q, \  \left\|q({a}/{q})^i\right\|_{\bR/\bZ}\leq \delta, \ i=2,\dots,m   \}.
\end{align}
This is the number of rational points in a $\delta/q$ neighborhood of $\Gamma_m:=(t, t^2, \ldots, t^m)$ over the interval $I_m$. \textit{Heuristically}, the above should be bounded from below by 
\begin{align}
    |I_m|\delta^{m-1}Q^2 =  \delta^{m-1+\frac{1}{m+1}}Q^{2-\frac{1}{m+1}}.
\end{align}
Compare this with 
the Example 1.3 of \cite{Bers12}
which draws the lower bound conclusion
$\delta \geq Q^{-1/c}$ for manifolds 
with co-dimension $1\leq c< n-1$
by essentially only looking 
at certain manifolds with many points 
on them.
The examples here produce a sequence of nested manifolds 
with many points \emph{near} them.

\subsection{Deducing the heuristic from existing work}
\label{subsec existing lit}
Rational points near most small intervals get their fair share 
of rational points from a non-degenerate curve. 
This is implicit in the work \cite{Bers12} of Beresnevich,
as well as the recent work \cite{ChowSrivastavaTechnauYu2026}.
Specifically, one can use 
Lemma 4.12 and Lemma 4.15
of \cite{ChowSrivastavaTechnauYu2026} 
with the following specializations. 
For a small $\eta>0$, let 
\begin{gather*}
Q^{-1/(2m+1)} \leq \delta \leq \frac12,
\quad
\delta_*:=\eta\delta, \quad 
T:=\left(\frac{\delta}{Q}\right)^{1/(m+1)},
\quad
\delta_1=\cdots=\delta_m=\delta_*,
\\
\Delta:=Q^{-1+4\eta},
\qquad
K:=Q^{-1/2+2\eta}\delta_*^{-1/2},
\qquad
r:=\frac14\left(\frac{\delta_*}{Q}\right)^{1/2}.
\end{gather*}
Furthermore, we need to formally restrict
the curve to the smaller interval 
$J=[T/200,T/100]$, in place of $I=[-1,1]$, to ensure
none of the coordinate functions vanishes on the interval,
which the set-up of \cite{ChowSrivastavaTechnauYu2026} 
requires at some points.
With these choices,
we can deduce the required bound from a routine smoothing/unsmoothing procedure 
and the cited lemmas.
(More precisely, we can deduce that the interval $J$ can be covered by a finite 
collection of $O(1/r)$ many small intervals of length $r$, such that the piece of the curve parametrized over most of them will get the
expected share of rational points near them. 
Hence, many 
different choices of such intervals $I_m$.)

In the next subsection we present 
an elementary, self-contained argument adapted specifically to the moment curve $\Gamma$.

\subsection{Making the Counter-Example Rigorous}
\label{subsec rigorous}
Instead of trying to use lower bounds from \cite{Bers12} to justify counting over the curve parametrized over a small interval $I_m$ of length 
$
T_m\asymp(\delta/Q)^{\frac{1}{m+1}},
$ we shall use some direct algebra to count rational points near $\Gamma_n=(t, t^2, \ldots, t^n)$, by recursively choosing rational points of the form $(q, a_1, a_2, \ldots, a_m, 0,\ldots, 0)$ such that:
\begin{itemize}
    \item The points $a_i/q$ lie in a \emph{non-isotropic} neighborhood, that is, with $q\asymp Q$ and $|a_i|\asymp QT^i$ for $1\leq i\leq m$.
    \item We start by choosing $a_{m-1}$ and $a_m$ first, and then \emph{recursively} select $O(\delta/T_m^2)$ many choices of $a_{m-2}, \ldots, a_1$ and $q$ such that 
    $$\left|\left(\frac{a_1}{q}\right)^i-\frac{a_i}{q}\right|\lesssim \frac{\delta}{q}\qquad \textrm{ for } 2\leq i\leq m\,.$$
\end{itemize}

\subsection{A Lower Bound}
\label{subsec lower}
As mentioned in the introduction, 
$N_{\Gamma}(\delta, Q) \lesssim \delta^{n-1} Q^2$ 
is known to be wrong by \cite{chen2026sharp}
in the range from \eqref{eq CSST range}.
Therefore, we now examine the range 
$$
\delta\in D_Q:=(c Q^{-1/(2n-3)},1)
$$
for a fixed $c>0$ on which the following 
implied constants may depend without further mention.
We shall show that  
if $2\leq m\leq n-1$ and $\delta \in D_Q$, then
\begin{equation}\label{eq lower bound count}
N_{\Gamma}(\delta, Q)\gtrsim 
\delta^{m-1+\frac{1}{m+1}}Q^{2-\frac{1}{m+1}}.
\end{equation}
First we demonstrate that 
\eqref{eq lower bound count} follows once we know the following 
lemma. 
\begin{lemma}\label{le lower bound}
Let $Q>1$ be a large. 
Given $\delta\in D_Q$, and $2 \leq m\leq n-1$, define 
$$
T=\left(\frac{\delta}{Q}\right)^{\frac{1}{m+1}}.
$$
Then, there exists a constant $K=K(\Gamma)\in (0,1)$ so that
$N_{\Gamma}(\delta, Q)\gtrsim \# A_m(K\delta,Q)$
where 
$$
A_m(\delta,Q):=\left\{(a_0, a_1, \ldots, a_m)\in \Z^{m+1}:
\begin{matrix}
  a_j\in [QT^j/2,QT^j] \,\, \mathrm{for\,}0\leq j\leq m, \\
  \left|a_i-\frac{a_{i+1}^2}{a_{i+2}}\right|\lesssim 
  \frac{\delta}{T^2}\,\, \mathrm{for\,}0\leq i\leq m-2
\end{matrix}
\right\}.
$$
\end{lemma}
Once the lemma is at hand, the deduction of \eqref{eq lower bound count} is almost immediate:
\begin{proof}[Proof of \eqref{eq lower bound count}]
Notice that 
$
\delta T^{-2}=(\delta^{\frac{m-1}{2}} Q)^{\frac{2}{m+1}}
\geq (Q^{1-\frac{n-2}{4n-6}})^{\frac{2}{m+1}}
$
tends to infinity for $\delta\in D_Q$ as $Q\rightarrow \infty$.
Let $\varepsilon>0$ be small.

We start by first choosing $a_m$ in $\asymp QT^m$ many ways
in the interval $QT^m \cdot (3/4-\varepsilon,3/4+\varepsilon)$
and then $a_{m-1}$ in $\asymp QT^{m-1}$ many ways
in $QT^{m-1} \cdot (3/4-\varepsilon,3/4+\varepsilon)$.
This allows us to pick the remaining $a_0,\ldots, a_{m-2}$
in $\asymp(\delta T^{-2})^{m-1}$ ways.
Here the point of choosing $a_m,a_{m-1}$
in these shorter intervals is to ensure that
$a_1/a_0$ remains in the interval $[-1,1]$.
We can check this
by formally setting $\varepsilon=0$ which would enforce
$a_j/QT^j=3/4$ for $j\in \{m,m-1\}$ and thus $a_j/QT^j\in[1/2,1]$
for all $1\leq j\leq m$.
The upshot is
$$
\# A_m(\delta,Q)\gtrsim Q^2 T^{2m-1} \left(\frac{\delta}{T^2}\right)^{m-1}=\delta^{m-1}Q^2T
=\delta^{m-1+\frac{1}{m+1}}Q^{2-\frac{1}{m+1}}. 
$$
Applying Lemma \ref{le lower bound} then completes the proof.
\end{proof}
Next, we establish  Lemma \ref{le lower bound}.
\begin{proof}[Proof of Lemma \ref{le lower bound}]
Let $2\leq i \leq m$.
Suppose for the moment we knew that
\begin{equation}\label{eq master rec}
   \frac{a_{j+1}^{i-j}}{a_{j}^{i-(j+1)}}=
\frac{a_{j+2}^{i-(j+1)}}{a_{j+1}^{i-(j+2)}}+O\left(\delta\right)
\end{equation}
for every $j=0,\ldots,i-2$.
Then applying \eqref{eq master rec} exactly $i-1$ many times
with $j=0,1,\ldots,i-3$ implies
$$
 \frac{a_1^i}{a_0^{i-1}}=\frac{a_{i-1}^2}{a_{i-2}}+O\left(\delta\right)=a_{i}+O\left(\delta\right).
$$
This implies 
$
\big\Vert a_0\big(\frac{a_1}{a_0}\big)^i\big\Vert 
\lesssim\delta
$
for $1\leq i\leq m$
while for $m<i\leq n$ the bound follows automatically
from $a_1/a_0=O(T)$.
This would then complete the proof by rescaling the
implied constants down to one, that is replacing $\delta$
with $\delta K$ for some small $K>0$.
By clearing the denominators in
\eqref{eq master rec}, we wish to show 
$$
a_{j+1}^{i-j+i-j-2}=
 a_{j}^{i-j-1} a_{j+2}^{i-j-1}+
 O\left(\delta a_{j+1}^{i-j-2} a_{j}^{i-j-1} \right).
$$
When $i-j=2$ then the choice of $a_j$ directly verifies
the relation. So, we assume $i-j\geq 3$.
Let $k=i-j-1$. Then, it suffices to show
\begin{equation}\label{eq reduction}
a_{j+1}^{2k}=
 a_{j}^{k} a_{j+2}^{k}+
 O\left(\delta a_{j+1}^{k-1} a_{j}^{k} \right).
\end{equation}
Further, we know that 
$
u_j:= a_{j+1}^2 - a_j a_{j+2}
$
satisfies $\vert u_j\vert \lesssim 
\frac{\delta}{T^2} a_{j+2} 
\lesssim \delta Q T^j \asymp \delta a_j$.
Now, 
$$
a_{j+1}^{2k}=
(u_j + a_j a_{j+2})^{k}
=
\sum_{\ell=0}^k 
\binom{k}{\ell} u_j^{\ell}
(a_j a_{j+2})^{k-\ell}.
$$
Since $a_{j+1}^2-a_j a_{j+2}=o(a_ja_{j+2})$,
we infer
$
a_{j+1}^{2k}= 
(a_j a_{j+2})^{k}
 + O(\delta a_j (a_ja_{j+2})^{k-1}).
$
Because $a_{j+2}=O(a_{j+1})$,
\eqref{eq reduction} follows, as required.
\end{proof}

\subsection{Deduction of Theorem \ref{thm main} 
from \eqref{eq lower bound count}}
The above family of examples shows that for $\delta^{n-1}Q^2$ to be a valid upper bound from $N_{\Gamma}$, we need for 
$\delta\in D_Q$ that
$$
\delta^{n-1}Q^2\gtrsim
\delta^{m-1+\frac{1}{m+1}}Q^{2-\frac{1}{m+1}}
$$
for all $2\leq m< n$. In fact, the lower-bound in the case
$m=1$ was already shown in \cite{chen2026sharp}.
Hence, we can investigate when
$ \delta^{n-m-\frac{1}{m+1}}\gtrsim Q^{-\frac{1}{m+1}}$
for every $1\leq m<n$,
or in other words
$ \delta\gtrsim Q^{-\frac{1}{(n-m)(m+1)-1}} $.
It remains to maximize $(n-m)(m+1)-1$ for integers 
$1\leq m<n$. To do so, we consider the critical point at $\frac{n-1}{2}$ of the polynomial 
$P(X)=(n-X)(X+1)-1=-X^2+(n-1)X+n-1$ just as in \cite{hickman2025counting} which then completes the argument.

\bibliography{ref}
\end{document}